\documentclass[reqno,12pt, oneside]{amsart}

\usepackage{enumerate}
\usepackage[nobysame]{amsrefs}
\usepackage{enumitem}
\usepackage{amssymb}
\usepackage{amsmath}
\usepackage{amsthm}
\usepackage{mathtools}
\usepackage{amscd}
\usepackage{microtype}
\usepackage[right=2.5cm,left=2.5cm, top=2.5cm, bottom=2.5cm,includehead]{geometry}
\usepackage{amsfonts}

\usepackage{tgpagella}
\usepackage{eulervm}

\usepackage[backgroundcolor=lightgray]{todonotes}

\theoremstyle{plain}
\newtheorem{theorem}{Theorem}
\newtheorem{corollary}[theorem]{Corollary}

\newtheorem{proposition}[theorem]{Proposition}

\theoremstyle{remark}
\newtheorem{remark}[theorem]{Remark}
\newtheorem{example}[theorem]{Example}

\theoremstyle{definition}
\newtheorem{definition}[theorem]{Definition}

\newtheorem*{question}{Open problem}
\newtheorem*{convention}{Important convention}

\DeclarePairedDelimiter{\abs}{\lvert}{\rvert}

\DeclarePairedDelimiter{\norm}{\lVert}{\rVert}

\DeclarePairedDelimiterX{\dset}[2]{\lbrace}{\rbrace}{#1\;\delimsize|\;#2}

\DeclareMathOperator{\conv}{conv}
\DeclareMathOperator{\clconv}{\overline{\conv}}

\DeclareMathOperator{\diam}{diam}

\newcommand{\ball}{{\overline{B}}}

\newcommand{\dotminus}{\mathbin{\smash{-\kern-0.55em\raisebox{0.5ex}{\scalebox{1.3}{$\cdot$}}\kern0.2em}}}
\newcommand{\N}{\mathbb{N}}

\title[Axiomatic measures of non-compactness]{On the independence of axioms for measures of non-compactness}

\author[J.Grzybowski]{Jerzy Grzybowski} 
\address{J. Grzybowski, Department of Nonlinear Analysis and Applied Topology\\
  Faculty of Mathematics and Computer Science\\
  Adam Mickiewicz University, Pozna\'n\\
  ul.\ Uniwersytetu Pozna\'nskiego 4\\
  61-614 Pozna\'n\\
  Poland}
\email{jgrz@amu.edu.pl} 

\author[P.Kasprzak]{Piotr Kasprzak} 
\address{P. Kasprzak, Department of Nonlinear Analysis and Applied Topology\\
  Faculty of Mathematics and Computer Science\\
  Adam Mickiewicz University, Pozna\'n\\
  ul.\ Uniwersytetu Pozna\'nskiego 4\\
  61-614 Pozna\'n\\
  Poland}
\email{kasp@amu.edu.pl} 

\author[P.Kasprzak]{Piotr Ma\'ckowiak} 
\address{Piotr Ma\'ckowiak, Department of Nonlinear Analysis and Applied Topology\\
  Faculty of Mathematics and Computer Science\\
  Adam Mickiewicz University, Pozna\'n\\
  ul.\ Uniwersytetu Pozna\'nskiego 4\\
  61-614 Pozna\'n\\
  Poland}
\email{piotr.mackowiak@amu.edu.pl} 

\date{\today}

\begin{document}

\begin{abstract} 
We investigate the independence of the axioms in the Banaś--Goebel approach to the theory of measures of non-compactness. We present a complete answer for four of the five axioms. For the remaining axiom, we provide a full solution in finite-dimensional spaces and a partial one in the infinite-dimensional case.
\end{abstract}

\maketitle
	
\section{Introduction}

The notion of a measure of non-compactness was introduced by Kuratowski in 1930 (see~\cite{Kuratowski}*{p.~303}). For several decades, it remained largely unnoticed. It re-emerged in 1950s and 1960s, when it became an important tool in operator theory and fixed point theory, particularly in the works of Darbo, Sadovskii, and others (see, for example, the classical papers~\cites{Darbo, GM65, GGM57, Sadovskii}). This period also marked the beginning of efforts to place various notions of non-compactness within a unified framework.

One of the earliest axiomatic approaches in this direction was due to Sadovskii (see~\cites{Sadovskii1968, Sadovskii1972}, cf. also~\cite{AKPRS}*{Definition~1.2.1}). In that setting, a set function $\nu$, defined on the family of all non-empty bounded subsets of a locally convex space, is called a measure of non-compactness if it is invariant under passage to the closed convex hull, that is,
\[
 \text{$\nu(\clconv{A})=\nu(A)$ for every bounded set $A$.}
\]
A different axiomatic approach in metric spaces was later proposed by Martin\'{o}n in~1990 (see~\cite{Martinon1990}). 

Another development in this direction is the notion of a quasimeasure of non-compactness introduced by Krukowski in~2016 (see~\cites{Krukowski2016, Krukowski2017}). In this approach, the set function is not required to be invariant under taking convex hulls. This results in a framework that seems too broad to be used directly in classical Darbo- or Sadovskii-type fixed point theorems. Nevertheless, certain fixed point results can still be obtained by combining quasimeasures with an additional control of non-convexity (see~\cite{Krukowski2017}*{Theorem~24} and~\cite{BugajewskiGulgowski2019}*{Theorem~4}). A concrete examples of quasimeasures were provided, for example, by Krukowski in~\cite{Krukowski2016} in the case of the space of bounded and continuous mappings, and by Bugajewski and Gulgowski in~\cite{BugajewskiGulgowski2019} in the case of the space of functions of bounded Jordan variation.

Among the various axiomatic approaches that have been proposed, the framework introduced by Banaś and Goebel in their 1980 monograph (see~\cite{BG80}) has proved particularly influential and is nowadays regarded as one of the standard approaches to the theory.

Measures of non-compactness, especially within the Banaś--Goebel framework, play an important role in fixed point theory as well as in the study of integral and differential equations, evolution equations, differential inclusions, semigroup theory, and dynamical systems. For a detailed overview and applications, we refer the reader to the monographs~\cites{AKPRS, ADL, MR3289625, MR3642943}.

While the classical literature is largely concerned with applications of measures of non-compactness, more recent work has focused on structural and qualitative aspects of the theory. This includes general representation results for abstract measures of non-compactness on Banach spaces (see~\cites{ChenCheng2023, ChengChengShenTuZhang2018}), as well as investigations into the existence of inequivalent measures (see~\cites{MalletParetNussbaum2011a, MalletParetNussbaum2011b, AbletChengChengZhang2019, BanasMartinon1992}). Other directions include questions of countable determination of the Kuratowski measure~\cite{ChenCheng2021} and the study of minimal sets associated with measures of non-compactness (see~\cite{DominguezBenavides1986} and~\cite{ADL}*{Chapter~III}). For a broader discussion of these topics, we refer to the recent survey~\cite{MR4686571}.

One aspect of the theory of axiomatic measures of non-compactness that is still not fully understood concerns the independence of the defining axioms. This issue is important from both theoretical and practical perspectives, as it clarifies whether the axiom system is minimal and, consequently, whether the assumptions used in applications are in some sense optimal.

In this paper we investigate the independence of the axioms in the Banaś--Goebel framework. We present a complete answer for four of the five axioms. In the case of the generalized Cantor intersection property, we restate a general result originally published in Chinese in~\cite{ChenChengHe2023} (see also Section~\ref{sec:v} for further details) and present an example that is essentially different from the one given therein. For the remaining axiom, we solve the problem completely in finite-dimensional spaces and obtain only partial results in the infinite-dimensional setting.

The paper is organized as follows. In Section~2 we collect notation and recall the definition of an axiomatic measure of non-compactness due to Banaś and Goebel. Section~3 reviews known properties of axiomatic measures of non-compactness and related set functions, and develops several new refinements, extensions, and equivalent reformulations of the axioms. Section~4 is devoted to the independence of the axioms. Finally, in Section~5 we present equivalent reformulations of the definition of an axiomatic measure of non-compactness in both finite- and infinite-dimensional settings, based on the results obtained in the previous sections. The paper concludes with an open question.

\section{Notation and basic definitions}

\subsection{Notation}

Throughout the paper, $E$ denotes a Banach space (over either the field of real or complex numbers) endowed with the norm $\norm{\cdot}_E$. For any point $x \in E$ and any $r>0$, we write $\ball_E(x,r)$ for the closed ball in $E$ with center $x$ and radius $r$. Furthermore, we denote by $\mathcal{B}_E$ the collection of all non-empty, bounded subsets of $E$, and by $\mathcal{K}_E$ its subfamily consisting of relatively compact sets. For two non-empty subsets $A, B$ of $E$ and $\lambda \in \mathbb R$ we define $A+B:=\dset{a+b}{a\in A,\ b\in B}$ and $\lambda A:=\dset{\lambda a}{a\in A}$. In particular, if $A=\{x\}$ is a singleton, then $A+B$ coincides with the translate $x+B$ of $B$. 

We will also use the Hausdorff distance $d_H$, defined for $A,B \in \mathcal B_E$ by
\[
 d_H(A,B):=\inf\dset[\big]{r>0}{\text{$A\subseteq B+\ball_E(0,r)$ and $B\subseteq A+\ball_E(0,r)$}}.
\]
It is well-known that $d_H$ is a pseudo-metric on $\mathcal B_E$, meaning that it is symmetric and satisfies the triangle inequality, but the remaining third condition is relaxed to $d_H(A, A) = 0$ for all $A \in \mathcal B_E$. Actually, it can be easily checked that $d_H(A,B)=0$ if and only if the closures of $A$ and $B$ coincide. 

\subsection{Axiomatic measures of non-compactness}
\label{sec:aMNCs}

Let us now recall the main object of our study, namely the definition of an axiomatic measure of non-compactness due to Bana\'s and Goebel.

\begin{definition}[cf.~\cite{BG80}*{Definition~3.1.3}]\label{def:axiomatic_MNC}
Let $E$ be a Banach space. A function $\mu \colon \mathcal B_E \to [0,+\infty)$ is called an \emph{axiomatic measure of non-compactness}, if
\begin{enumerate}[label=\textup{(A\arabic*)}]
 \item\label{def:axiomatic_MNC_i} the family $\ker \mu  \coloneqq \dset{A \in \mathcal B_E}{\mu(A)=0}$, called the \emph{kernel} of the measure $\mu$, is non-empty, and is included in $\mathcal K_E$,
\end{enumerate}
and for any $\lambda  \in [0,1]$ and any $A,B, A_n \in \mathcal B_E$, where $n \in \mathbb N$, it satisfies the conditions\textup:
\begin{enumerate}[label=\textup{(A\arabic*)}, resume]
 \item\label{def:axiomatic_MNC_ii} if $A \subseteq B$, then $\mu(A)\leq \mu(B)$,

 \item\label{def:axiomatic_MNC_iii} $\mu(\conv A)=\mu(A)$,

 \item\label{def:axiomatic_MNC_iv} $\mu(\lambda A + (1-\lambda)B)\leq \lambda \mu(A) + (1-\lambda)\mu(B)$,

 \item\label{def:axiomatic_MNC_v} if the sets $A_n$ are closed and $A_{n+1}\subseteq A_n$ for $n \in \mathbb N$ with $\mu(A_n)\to 0$, then their intersection $\bigcap_{n=1}^\infty A_n$ is non-empty.
\end{enumerate}
\end{definition}

Axioms~\ref{def:axiomatic_MNC_ii} and~\ref{def:axiomatic_MNC_iv} are commonly referred to as the \emph{monotonicity} and \emph{convexity} of the measure $\mu$, respectively, while axiom~\ref{def:axiomatic_MNC_v} is known as the \emph{generalized Cantor \textup(intersection\textup) theorem}.

The simplest examples of axiomatic measures of non-compactness are the diameter
\[
 A\mapsto \diam A:=\sup_{x,y \in A}\norm{x-y}_E
\]
and the norm function
\[
 A \mapsto \norm{A}:=\sup_{x\in A}\norm{x}_E.
\]
Of course, the classical measures of Kuratowski, Hausdorff, and Istr\v{a}\c{t}escu also fit into the axiomatic framework. (For more information on those measures, see, for example,~\cites{ADL, AKPRS, BG80, BBK}.)

\begin{convention}
Throughout the paper, we distinguish between axiomatic measures of non-compactness and
real-valued set functions defined on $\mathcal{B}_E$ that satisfy only some of the axioms. The former will always be denoted by $\mu$, while the latter will be denoted by $\nu$. Also, by a slight abuse of terminology, we will refer to the set $\ker \nu$ as the \emph{kernel}, even when $\nu$ is not an axiomatic measure of non-compactness.
\end{convention}

\section{Preliminary results}

When introducing the axiomatic system, Bana\'s and Goebel already observed that
axiomatic measures of non-compactness are continuous with respect to the
Hausdorff distance (cf.~\cite{BG80}*{Theorem~3.2.2}). This property, in
fact, extends to a broader class of set functions.

\begin{proposition}\label{prop:continuity_nu}
Let $E$ be a Banach space, and let $\nu \colon \mathcal B_E \to [0,+\infty)$ be a set function satisfying axiom~\ref{def:axiomatic_MNC_ii}. Moreover, assume that there exists $\lambda \in (0,1)$ such that
\begin{equation}\label{eq:iv}
  \text{$\nu(\lambda A + (1-\lambda) B)\leq \lambda \nu(A) + (1-\lambda)\nu(B)$ for all $A,B \in \mathcal B_E$.}
\end{equation}
Then, $\nu$ is continuous with respect to the
Hausdorff distance, meaning that if $(A_n)_{n\in\mathbb N}$ is a sequence of non-empty, bounded subsets of $E$ converging in the Hausdorff distance to a non-empty, bounded set $A$, then $\lim_{n\to\infty} \nu(A_n) = \nu(A)$.
\end{proposition}

\begin{proof}
First, note that if the inequality $\nu(\lambda A + (1-\lambda) B)\leq \lambda \nu(A) + (1-\lambda)\nu(B)$ holds for all $A,B \in \mathcal B_E$, then it also remains valid
when $\lambda$ is replaced by $\lambda^k$ for any $k \in \mathbb N$. To see this, assume that for some fixed $k \in \mathbb N$ we already have $\nu(\lambda^k A + (1-\lambda^k) B)\leq \lambda^k \nu(A) + (1-\lambda^k)\nu(B)$. Then, for arbitrary $A,B \in \mathcal B_E$ we have 
\[
 \lambda^{k+1} A + (1-\lambda^{k+1})B \subseteq \lambda\bigl(\lambda^k A + (1-\lambda^k)B\bigr) + (1-\lambda)B.
\]
Therefore, by axiom~\ref{def:axiomatic_MNC_ii} and the inductive assumption, we obtain
\begin{align*}
\nu(\lambda^{k+1} A+(1-\lambda^{k+1})B)
 & \leq \nu\bigl(\lambda\bigl(\lambda^k A+(1-\lambda^k)B\bigr)+(1-\lambda)B\bigr)\\
 & \leq \lambda\nu(\lambda^k A+(1-\lambda^k)B)+(1-\lambda)\nu(B)\\
 & \leq \lambda\bigl(\lambda^k \nu(A)+(1-\lambda^k)\nu(B)\bigr)+(1-\lambda)\nu(B)\\
 & =\lambda^{k+1}\nu(A) + (1-\lambda^{k+1})\nu(B).
\end{align*}
This establishes the statement for $k+1$, and hence for all $k \in \mathbb{N}$.

Now, let $(A_n)_{n\in\mathbb N}$ be a sequence of non-empty, bounded subsets of $E$ converging in the Hausdorff distance to a non-empty, bounded set $A$. Choose $r>0$ such that $A\subseteq \ball_E(0,r)$. Our goal is to show that the sequence of real numbers $(\nu(A_n))_{n\in\mathbb N}$ converges to $\nu(A)$. To this end, we apply the \emph{subsequence principle}: a sequence $(a_n)_{n \in \mathbb N}$ converges to $a$ if and only if every subsequence $(a_{n_k})_{k \in \mathbb N}$ contains a further subsequence $(a_{n_{k_l}})_{l \in \mathbb N}$ converging to $a$. We now proceed as follows. Let $(\nu(A_{n_k}))_{k\in\mathbb N}$ be an arbitrary subsequence of $(\nu(A_n))_{n\in\mathbb N}$. We extract a
further subsequence $(\nu(A_{n_{k_l}}))_{l\in\mathbb N}$ in such a way that $d_H(A,A_{n_{k_l}}) < \lambda^l$ for all $l \in \mathbb N$. This is clearly possible since $d_H(A,A_n) \to 0$ as $n \to +\infty$. From the definition of the Hausdorff distance, it follows that for every $l \in \mathbb N$ we have       
\begin{align*}
 A_{n_{k_l}} &\subseteq A+\lambda^l \ball_E(0,1) \subseteq (1-\lambda^l)A + \lambda^l\bigl(A+\ball_E(0,1)\bigr)\\
& \subseteq (1-\lambda^l)A + \lambda^l \ball_E(0,r+1) \subseteq \ball_E(0,r+1)
\end{align*}
and, similarly,
\begin{align*}
 A &\subseteq A_{n_{k_l}} + \lambda^l \ball_E(0,1) \subseteq (1-\lambda^l)A_{n_{k_l}} + \lambda^l\bigl(A_{n_{k_l}}+ \ball_E(0,1)\bigr)\\
 & \subseteq (1-\lambda^l)A_{n_{k_l}} + \lambda^l \ball_E(0,r+2).
\end{align*}
Hence, combining the above inclusions with axiom~\ref{def:axiomatic_MNC_ii}
and the first part of the proof, we obtain
\[
 \nu(A_{n_{k_l}}) \leq (1-\lambda^l)\nu(A) + \lambda^l \nu(\ball_E(0,r+1)) 
\]
and
\[
  \nu(A) \leq (1-\lambda^l) \nu(A_{n_{k_l}}) + \lambda^l \nu(\ball_E(0,r+2)).
\]
Letting $l \to +\infty$, we conclude that $\nu(A_{n_{k_l}}) \to \nu(A)$. By the subsequence principle, this completes the proof.
\end{proof}

\begin{remark}\label{rem:axiom_closure_remark}
One immediate consequence of Proposition~\ref{prop:continuity_nu} is that every axiomatic measure of non-compactness is invariant under taking closures, that is, $\mu(\overline{A}) = \mu(A)$ for all $A \in \mathcal B_E$. Although this fact was already observed by Bana\'s and Goebel (see~\cite{BG80}*{p.~13}), many authors include it explicitly among the axioms for the sake of simplicity (see, for example,~\cites{MR3289625, MR3642943}). 
\end{remark}

Now that we have discussed a property of axiomatic measures of non-compactness
which is not explicitly included among the axioms in Definition~\ref{def:axiomatic_MNC}, we return to the natural order in which the axioms are introduced.

In all works on axiomatic measures of non-com\-pact\-ness known to the authors, the kernel of the measure $\mu$ in axiom~\ref{def:axiomatic_MNC_i} is always assumed to be a subset of $\mathcal K_E$. Surprisingly, this property already follows from axioms~\ref{def:axiomatic_MNC_ii} and~\ref{def:axiomatic_MNC_v}.

\begin{proposition}\label{prop:axiom_i_remark}
Let $E$ be a Banach space, and let $\nu \colon \mathcal B_E \to [0,+\infty)$ be a set function that satisfies axioms~\ref{def:axiomatic_MNC_ii} and~\ref{def:axiomatic_MNC_v}. Then, $\ker \nu \subseteq \mathcal K_E$. 
\end{proposition}

\begin{proof}
If the Banach space $E$ is finite-dimensional, the claim is immediate, since in this case $\mathcal B_E = \mathcal K_E$. Thus, we may assume that $E$ is infinite-dimensional. Suppose, on the contrary, that there exists a set $A \in \mathcal B_E \setminus \mathcal K_E$ such that $\nu(A)=0$. Since $A$ is not relatively compact, we can choose a sequence $(x_n)_{n \in \mathbb N}$ of points in $A$ with the property that none of its subsequences converges in $E$. For each $n \in \mathbb N$, let $A_n:=\{x_n, x_{n+1}, \ldots\}$. Clearly, these sets form a non-increasing sequence. They are also closed, since otherwise the sequence $(x_n)_{n \in \mathbb N}$ would admit a convergent subsequence. By axiom~\ref{def:axiomatic_MNC_ii}, we further obtain $0 \leq \nu(A_n) \leq \nu(A) = 0$ for all $n \in \mathbb N$. Hence, axiom~\ref{def:axiomatic_MNC_v} ensures that the intersection $\bigcap_{n=1}^\infty A_n$ is non-empty. However, any point in this intersection would serve as a limit of a subsequence of $(x_n)_{n \in \mathbb N}$, contradicting our assumption that the sequence has no convergent subsequences. Therefore, $\ker \nu \subseteq \mathcal K_E$.
\end{proof}

\begin{remark}\label{rem:axiom_ii_remark}
Note that axiom~\ref{def:axiomatic_MNC_ii} can be restated simply as
\[
 \text{$\nu(A \cup B) \geq \max\{\nu(A), \nu(B)\}$  for any $A, B \in \mathcal{B}_E$.}
\]
This inequality is generally strict. For example, in any Banach space $E$, take any elements $x \neq y$ and define $A:=\{x\}$ and $B:=\{y\}$. Then, using the diameter function, we have $\diam A=\diam B=0$, whereas $\diam(A\cup B)>0$.
\end{remark}

Extending the reasoning used in the first part of the proof of Proposition~\ref{prop:continuity_nu} leads to a result that may be of independent interest.

\begin{proposition}\label{prop:convexity_product}
Let $E$ be a Banach space, and let $\nu\colon \mathcal B_E \to [0,+\infty)$ be a set function satisfying axiom~\ref{def:axiomatic_MNC_ii}. Moreover, assume that condition~\eqref{eq:iv} holds for parameters $\lambda_1,\lambda_2 \in [0,1]$ in place of $\lambda$. Then,~\eqref{eq:iv} also holds with $\lambda:=\lambda_1 \cdot \lambda_2$. 
\end{proposition}

Interestingly, as in the classical setting of real-valued functions, midpoint convexity is equivalent to convexity, provided a mild regularity condition is imposed.

\begin{proposition}\label{ax:iv_product}
Let $E$ be a Banach space, and let $\nu \colon \mathcal B_E \to [0,+\infty)$ be a set function satisfying axiom~\ref{def:axiomatic_MNC_ii}. Then, the following conditions are equivalent\textup:
\begin{enumerate}[label=\textup{(\alph*)}, itemsep=2pt]
 \item\label{ax:iv_product_a} $\nu(\lambda A + (1-\lambda)B)\leq \lambda \nu(A) + (1-\lambda)\nu(B)$ for all $A,B \in \mathcal B_E$ and \emph{all} $\lambda \in [0,1]$,

 \item\label{ax:iv_product_b} $\nu(\lambda A + (1-\lambda)B)\leq \lambda \nu(A) + (1-\lambda)\nu(B)$ for all $A,B \in \mathcal B_E$ and \emph{some} $\lambda \in (0,1)$,

 \item $\nu(\frac{1}{2} A + \frac{1}{2} B)\leq \frac{1}{2} \nu(A) + \frac{1}{2}\nu(B)$ for all $A,B \in \mathcal B_E$.
\end{enumerate}
\end{proposition}

\begin{proof}
Clearly, it suffices to prove the implication~$\ref{ax:iv_product_b} \Rightarrow~\ref{ax:iv_product_a}$.

Let us fix an arbitrary $\lambda \in (0,1)$, and assume that there exists $\alpha \in (0,1)$ such that  
\begin{equation}\label{eq:iv_i}
 \text{$\nu(\alpha A + (1-\alpha)B)\leq \alpha \nu(A) + (1-\alpha)\nu(B)$ for all $A,B \in \mathcal B_E$}.
\end{equation}
Clearly, we may assume that $\lambda \neq \alpha$; otherwise there is nothing to prove. Moreover, by symmetry of~\eqref{eq:iv_i}, we may further assume without loss of generality that $\alpha > \lambda$.

We now construct a non-increasing sequence $(\alpha_n)_{n\in \mathbb N}$, whose terms are finite products of $\alpha$ and factors of the form $1-\alpha^j$, where $j\in\mathbb N$, and which converges to $\lambda$. To begin, set $\alpha_0:=\alpha$ and for each $n\in \mathbb N \cup \{0\}$, define
\[
k_{n+1}:=\min\dset{j\in \mathbb N}{\alpha_n(1-\alpha^j)>\lambda},
\qquad
\alpha_{n+1}:=\alpha_n(1-\alpha^{k_{n+1}}).
\]
Note that for every $n\in \mathbb N$ we have
\[
\lambda < \alpha_n(1-\alpha^{k_{n+1}})
= \alpha_{n-1}(1-\alpha^{k_n})(1-\alpha^{k_{n+1}})
< \alpha_{n-1}(1-\alpha^{k_{n+1}}).
\]
By the minimality of $k_n$, this yields $k_{n+1}\geq k_n$. Hence, the sequence $(k_n)_{n\in\mathbb N}$ is non-decreasing. 

We next show that $(k_n)_{n\in\mathbb N}$ is unbounded. Suppose, for a contradiction, that it is bounded. Since $(k_n)_{n\in\mathbb N}$ is a non-decreasing sequence of positive integers, it must eventually stabilize. In other words, there exists $N\in\mathbb N$ such that $k_{N+n}=k_N$ for all $n\geq 0$. In this case, for $n\geq 0$ we obtain
\[
 \alpha_{N+n}=\prod_{j=1}^{N+n} \alpha_0 (1-\alpha^{k_j})=(1-\alpha^{k_N})^n \cdot \prod_{j=1}^{N} \alpha_0 (1-\alpha^{k_j}). 
\]
Since $0<1-\alpha^{k_N}<1$, it follows that $\alpha_{N+n}\to 0$ as $n\to +\infty$. This contradicts the fact that $\alpha_n>\lambda>0$ for all $n\in\mathbb N$. Therefore, $(k_n)_{n\in\mathbb N}$ must be unbounded.

Finally, we prove that the sequence $(\alpha_n)_{n \in \mathbb N}$ converges to $\lambda$.
Since $(\alpha_n)_{n\in\mathbb N}$ is decreasing and bounded from below by $\lambda$, it has a limit $a := \inf_{n \in \N} \alpha_n \geq \lambda$. Assume, for a contradiction, that $a > \lambda$. Since $\alpha \in (0,1)$, we can choose $l \in \mathbb N$ such that $a(1-\alpha^l) > \lambda$. By the definition of the infimum, it follows that $\alpha_n(1-\alpha^l) > \lambda$ for all $n \in \N$. In particular, this would imply that $k_n \leq l$ for all $n \geq 2$, contradicting the fact that the sequence $(k_n)_{n\in\mathbb N}$ is unbounded. Therefore, $\alpha_n \to \lambda$ as $n \to +\infty$.

We also observe that the sequence $(\alpha_n)_{n \in \mathbb{N}}$ constructed above has an interesting property. Namely, by Proposition~\ref{prop:convexity_product} and symmetry, condition~\eqref{eq:iv_i} continues to hold upon replacing $\alpha$ with $\alpha_n$, for every $n \in \mathbb{N}$.

Now, let $A,B \in \mathcal{B}_E$, and let $r>0$ be such that $A,B \subseteq \ball_E(0,r)$. Then, for any $n \in \mathbb{N}$ and any $a \in A$, $b \in B$, we have
\begin{align*}
 \lambda a + (1-\lambda)b
 &= \alpha_n a + (1-\alpha_n)b + (\alpha_n-\lambda)(b-a),
\end{align*}
and similarly,
\begin{align*}
 \alpha_n a + (1-\alpha_n)b
 &= \lambda a + (1-\lambda)b + (\lambda-\alpha_n)(b-a).
\end{align*}
It follows that
\[
 \lambda A + (1-\lambda)B \subseteq \alpha_n A + (1-\alpha_n) B + (\alpha_n-\lambda)\, \ball_E(0,2r),
\]
and likewise,
\[
 \alpha_n A + (1-\alpha_n)B \subseteq \lambda A + (1-\lambda)B + (\alpha_n-\lambda)\, \ball_E(0,2r).
\]
Hence,
\[
 d_H\bigl(\lambda A + (1-\lambda)B,\ \alpha_n A + (1-\alpha_n)B\bigr) \leq 2r(\alpha_n-\lambda) \to 0
 \quad \text{as } n \to +\infty.
\]
By~\eqref{eq:iv_i} applied to each term of the sequence $(\alpha_n)_{n \in \mathbb{N}}$ separately, in view of Proposition~\ref{prop:continuity_nu}, we therefore obtain
\begin{align*}
 \nu(\lambda A + (1-\lambda)B)
 &= \lim_{n \to \infty} \nu(\alpha_n A + (1-\alpha_n)B) \\
 &\leq \lim_{n \to \infty} \bigl(\alpha_n \nu(A) + (1-\alpha_n)\nu(B)\bigr) = \lambda \nu(A) + (1-\lambda)\nu(B).
\end{align*}
This completes the proof.
\end{proof}

We conclude this section with two remarks concerning axiom~\ref{def:axiomatic_MNC_v}.

\begin{remark}\label{rem:axiom_v_remark}
It is a well-known fact that the non-empty intersection $\bigcap_{n=1}^\infty A_n$ in axiom~\ref{def:axiomatic_MNC_v} belongs to the kernel of the measure $\mu$, although some texts list it as an additional requirement. Indeed, since $\bigcap_{n=1}^\infty A_n \subseteq A_k$ for every $k \in \mathbb{N}$, by axiom~\ref{def:axiomatic_MNC_ii} we have $\mu(\bigcap_{n=1}^\infty A_n)\leq \mu(A_k) \to 0$ as $k \to +\infty$, and therefore $\mu(\bigcap_{n=1}^\infty A_n)=0$. 
\end{remark}

\begin{remark}\label{rem:axiom_vv_remark}
In the original formulation by Bana\'s and Goebel, as well as in many later works, axiom~\ref{def:axiomatic_MNC_v} is stated without any restriction on the dimension of $E$. Note, however, that in finite-dimensional spaces the axiom is automatically satisfied.
\end{remark}

\section{Independence of the axioms}

We say that an axiom is \emph{independent} of the others if in every Banach space there exists a set function which satisfies all the axioms in Definition~\ref{def:axiomatic_MNC}, except for that particular one.

\subsection{Axiom~\ref{def:axiomatic_MNC_i}} 
Axiom~\ref{def:axiomatic_MNC_i} is independent of the others. Indeed, consider the constant set function $\nu\colon \mathcal B_E \to [0,+\infty)$ defined by $\nu(A) := 1$ for all $A\in \mathcal B_E$; it never vanishes on a non-empty bounded set, yet all other axioms are satisfied, with axiom~\ref{def:axiomatic_MNC_v} holding vacuously.

A brief comment is in order. Formally, axiom~\ref{def:axiomatic_MNC_i} is the conjunction of two conditions, which might suggest the need for separate examples addressing each condition individually. However, in view of Proposition~\ref{prop:axiom_i_remark}, a single example is sufficient.

\subsection{Axiom~\ref{def:axiomatic_MNC_ii}} 
\label{sec:axiom_ii}
To show that axiom~\ref{def:axiomatic_MNC_ii} is independent of the others let us consider the set function $\nu \colon \mathcal B_{E} \to [0,+\infty)$ given by
\[
 \nu(A):=(\diam A)^2+\inf_{a \in \conv A}\norm{a}_E \ \ \text{for $A\in \mathcal B_E$}.
\]
The kernel of $\nu$ is non-empty and included in $\mathcal K_E$, as it consists of a singleton $\{0\}$. Hence, $\nu$ satisfies axiom~\ref{def:axiomatic_MNC_i}. On the other hand, $\nu$ fails to satisfy axiom~\ref{def:axiomatic_MNC_ii}. To see this, just take the sets $A:=\{x\}$ and $B:=\{0,x\}$, where $x \in E$ is a fixed point of norm $\frac{1}{2}$. Although $A\subseteq B$, we have $\nu(A)=\frac{1}{2}$ and $\nu(B)=\frac{1}{4}$. Axiom~\ref{def:axiomatic_MNC_iii} is satisfied trivially, since $\conv(\conv A)=\conv A$ for any non-empty set, and the diameter function itself is an axiomatic measure of non-compactness. To check that $\nu$ satisfies~\ref{def:axiomatic_MNC_iv}, we use a simple identity:
\begin{equation}\label{eq:convex}
\conv(\lambda A + (1-\lambda)B) = \lambda \conv A + (1-\lambda)\conv B \ \ \text{for $A,B \in \mathcal B_E$ and $\lambda \in [0,1]$}. 
\end{equation}
Any point $c \in \conv(\lambda A + (1-\lambda)B)$ can be written as $c=\lambda a + (1-\lambda) b$ for some $a \in \conv A$ and $b \in \conv B$. Moreover, $\norm{c}_E \leq \lambda \norm{a}_E + (1-\lambda) \norm{b}_E$. Now, taking the infimum over $a$ and $b$, and using the fact that square functions are convex and non-decreasing on the non-negative half-axis, we obtain
\[
 \nu(\lambda A + (1-\lambda)B) \leq \lambda \nu(A) + (1-\lambda) \nu(B).
\]
Finally, let us consider a non-increasing sequence $(A_n)_{n \in \mathbb{N}}$ of non-empty, bounded, and closed subsets such that $\nu(A_n)\to 0$. In this case we also have $\diam A_n \to 0$. By the classical 
Cantor intersection theorem (see, for instance, \cite{BBK}*{Corollary~1.1.33} or \cite{Kuratowski}*{Th\'eor\`eme~1}), it follows that the intersection 
$\bigcap_{n=1}^\infty A_n$ is non-empty, meaning that $\nu$ satisfies axiom~\ref{def:axiomatic_MNC_v}.

\subsection{Axiom~\ref{def:axiomatic_MNC_iii}} 
In contrast with the two preceding axioms, axiom~\ref{def:axiomatic_MNC_iii} requires a more delicate analysis. While we will not be able to fully resolve the question of its independence, as we shall see, in certain situations it can in fact be derived from the remaining axioms. To this end, we introduce the following notion, motivated by Carath\'eodory’s theorem on the structure of convex hulls in finite-dimensional Euclidean spaces (see, for example,~\cite{Schneider}*{Theorem~1.1.4}). Although the connection with this classical result may not be apparent at first sight, it will become clear in due course.

\begin{definition}
We say that a non-empty subset $A$ of a Banach space has the 
\emph{weak Carath\'eodory property} if there exists $\lambda \in (0,1]$ such that $\lambda A + (1-\lambda)\conv A = \conv A$.
\end{definition}

Before considering examples of sets with the weak Carath\'eodory property, we first pause to examine the definition.

\begin{remark}\label{rem:wCp_equivalent_1}
Since for every $\lambda \in (0,1]$ we always have the inclusion
\[
 \lambda A + (1-\lambda)\conv A \subseteq \conv A,
\]
we can equivalently define the weak Carath\'eodory property as the existence
of some $\lambda \in (0,1]$ for which
\[
\conv A \subseteq \lambda A + (1-\lambda)\conv A.
\]
\end{remark}

\begin{remark}
Note also that allowing longer convex combinations in the definition of the weak
Ca\-ra\-th\'e\-o\-do\-ry property does not lead to a larger class of sets. Indeed, suppose
there exist $\lambda_1, \ldots, \lambda_m \in (0,1]$ and $\lambda_{m+1},\ldots,\lambda_n \in [0,1]$ with $\lambda_1 + \cdots + \lambda_n = 1$ such that
\[
 \conv A \subseteq \lambda_1 A + \lambda_2 A + \cdots + \lambda_m A + \lambda_{m+1}\conv A + \cdots +\lambda_n \conv A,
\]
Then,
\begin{align*}
  \conv A &\subseteq \lambda_1 A + \lambda_2 A + \cdots + \lambda_m A + \lambda_{m+1}\conv A + \ldots +\lambda_n \conv A\\
	& \subseteq \lambda_1 A + \lambda_2 \conv A + \ldots +\lambda_n \conv A\\
	& = \lambda_1 A + (\lambda_2 + \cdots +\lambda_n) \conv A,
\end{align*}
which shows that $A$ has the weak Carath\'eodory property.

Conversely, if $A$ has the weak Carath\'eodory property, then
\begin{align*}
 \conv A &\subseteq \lambda A + (1-\lambda)\conv A\\
& \subseteq \lambda A + (1-\lambda)\lambda A + (1-\lambda)^2 \conv A\\
& \subseteq \lambda A + (1-\lambda)\lambda \conv A + (1-\lambda)^2 \conv A,
\end{align*}
and, inductively,
\[
 \conv A\subseteq \lambda A + (1-\lambda)\lambda \conv A + \cdots + (1-\lambda)^{n-1}\lambda \conv A +(1-\lambda)^n \conv A.
\]
\end{remark}

With these remarks in place, we now proceed to a proposition that offers an alternative but equivalent formulation of the definition of the weak Carath\'eodory property.

\begin{proposition}\label{prop:wCp_equivalent}
A non-empty subset $A$ of a Banach space has the weak Carath\'eodory property with parameter $\lambda \in (0,1]$ if and only if every element of $\conv A$ has a representation 
\[
 \sum_{i=1}^n \lambda_i a_i, \quad a_i \in A,\ \lambda_i \geq 0,\ \sum_{i=1}^n \lambda_i=1
\]
such that $\lambda_j \geq \lambda$ for some index $j \in \{1,\ldots,n\}$.
\end{proposition}

\begin{proof}
First, assume that $A$ has the weak Carath\'eodory property with parameter $\lambda$, and take any $x \in \conv A$. Then, there exist $y \in A$ and $z \in \conv A$ such that $x=\lambda y + (1-\lambda)z$. Moreover, we can represent $z$ in the form
\[
 z= \sum_{i=1}^n \kappa_i a_i,
\] 
where $a_i \in A$, $\kappa_i \geq 0$ and $\sum_{i=1}^n \kappa_i=1$. Hence,
\[
 x = \lambda y + \sum_{i=1}^n (1-\lambda)\kappa_i a_i.
\]
And in this convex combination the first coefficient is exactly $\lambda$.

Now, to prove the opposite implication let us take any $x \in \conv A$ and assume that we can write it as
\[
 x=\sum_{i=1}^n \lambda_i a_i
\]
for some $a_i \in A$, $\lambda_i \geq 0$, $\sum_{i=1}^n \lambda_i=1$ with $\lambda_1 \geq \lambda$. Then,
\[
 x=\lambda a_1 + (1-\lambda)\Biggl( \frac{\lambda_1-\lambda}{1-\lambda}a_1 + \sum_{i=2}^n \frac{\lambda_i}{1-\lambda} a_i\Biggr).
\]
Since the latter element belongs to $\conv A$, we obtain $\conv A \subseteq \lambda A + (1-\lambda)\conv A$. This proves that $A$ has the weak Carath\'eodory property with parameter $\lambda$.
\end{proof}

An immediate consequence of Proposition~\ref{prop:wCp_equivalent} is the following corollary.

\begin{corollary}\label{cor:weak_c_property_decreasing}
If a non-empty set has the weak Carath\'eodory property for some parameter $\lambda \in (0,1]$, then it also enjoys this property for every $\kappa \in (0,\lambda]$.
\end{corollary}

Clearly, every convex set has the weak Carath\'eodory property. As we will see shortly, every finite-dimensional set -- that is, any set contained in a finite-dimensional subspace of $E$ -- also has the weak Carath\'eodory property (see Proposition~\ref{prop:Cp_implies_wCp}). A method for constructing new sets with the weak Carath\'eodory property from given ones is presented in the following proposition.

\begin{proposition}\label{prop:weak_caratheodory_properties}
If non-empty subsets $A,B$ of the Banach space have the weak Carath\'eodory property, then so do $A \cup B$ and $A + B$.
\end{proposition}

\begin{proof}
Assume that the sets $A$ and $B$ have the weak Carath\'eodory property. Then, by Remark~\ref{rem:wCp_equivalent_1} and Corollary~\ref{cor:weak_c_property_decreasing}, there exists $\lambda \in (0,1]$ such that
\[
\conv A \subseteq \lambda A + (1-\lambda)\conv A
\quad \text{and} \quad
\conv B \subseteq \lambda B + (1-\lambda)\conv B.
\]
Consequently,
\begin{align*}
 \conv(A+B) &= \conv A + \conv B \\
 &\subseteq \lambda A + (1-\lambda)\conv A + \lambda B + (1-\lambda)\conv B \\
 &= \lambda(A+B) + (1-\lambda)\conv(A+B),
\end{align*}
which shows that $A+B$ has the weak Carath\'eodory property.

We now turn to the case of the union $A \cup B$.  Since the sets $\conv A$ and $\conv B$ are convex, it can be shown that 
\[
 \tfrac{1}{2}\bigl(\conv A \cup \conv B\bigr) + \tfrac{1}{2}\conv\bigl(\conv A \cup \conv B\bigr)=\conv\bigl(\conv A \cup \conv B\bigr)
\]
(we omit the proof here, as a more general result will be established in Proposition~\ref{prop:caratheodory_properties}). Note also that
\[
 \conv A \cup \conv B \subseteq \lambda(A \cup B) + (1-\lambda) \conv(A\cup B) \subseteq \conv\bigl(\conv A \cup \conv B\bigr)
\]
and
\begin{align*}
& \conv\bigl( \lambda(A \cup B) + (1-\lambda) \conv(A\cup B)\bigr)\\
&\qquad = \lambda \conv (A\cup B) + (1-\lambda)\conv (A\cup B)=\conv (A\cup B).
\end{align*}
Therefore,
\begin{align*}
  &\conv(A \cup B)\\
	   &\quad = \conv\bigl(\conv A \cup \conv B\bigr)\\
	   &\quad \subseteq \tfrac{1}{2}\lambda(A \cup B) + \tfrac{1}{2}(1-\lambda) \conv(A\cup B)+\tfrac{1}{2}\conv\bigl( \lambda(A \cup B) + (1-\lambda) \conv(A\cup B)\bigr)\\
		&\quad=\tfrac{1}{2}\lambda(A \cup B) + \tfrac{1}{2}(1-\lambda) \conv(A\cup B)+\tfrac{1}{2}\conv(A\cup B)\\
		&\quad =\tfrac{1}{2}\lambda(A \cup B) + (1-\tfrac{1}{2}\lambda) \conv(A\cup B)
\end{align*}
Hence, the union $A\cup B$ has the weak Carath\'eodory property.
\end{proof}

We postpone the question of whether the weak Carath\'eodory property is preserved under intersections or under the Minkowski difference of sets to Remark~\ref{rem:minkowski_difference_intersection}.

\begin{theorem}\label{thm:no_weak_C}
In every infinite-dimensional Banach space there exists a non-empty and bounded set that does not have the weak Carath\'eodory property. 
\end{theorem}

\begin{proof}
It is well-known that every infinite-dimensional Banach space $E$ contains a closed infinite-dimensional Banach subspace $F$ which admits a Schauder basis $(e_n)_{n \in \mathbb N}$ consisting of unit vectors (see, for example,~\cite{Diestel1984}*{Corollary~3, p.~39}). 

Now, define $A:=\dset{e_n \in F}{n \in \mathbb N}$. It is clear that $A$ is bounded in $E$. Moreover, its convex hull is contained entirely in $F$. Fix an arbitrary positive integer $m$, and consider
\[
 x:=\frac{1}{2m}\sum_{n=1}^{2m} e_n \in \conv A.
\]
Since $x\in F$, the uniqueness of the Schauder expansion implies that the expression above is the only representation of $x$ as a convex combination of elements of $A$. Consequently, $x$ admits no convex representation in which at least one coefficient is greater than or equal to $\frac{1}{m}$. As $m$ was arbitrary, it follows from Proposition~\ref{prop:wCp_equivalent} and Corollary~\ref{cor:weak_c_property_decreasing} that $A$ does not have the weak Carath\'eodory property.
\end{proof}

\begin{remark}
The assumption that the Banach space $E$ is infinite-dimensional is essential. Indeed, it will follow from Proposition~\ref{prop:Cp_implies_wCp} that every subset of a finite-dimensional space has the weak Carath\'eodory property.
\end{remark}

\begin{remark}\label{rem:no_weak_C}
From the proof of Theorem~\ref{thm:no_weak_C}, it follows that the set $A:=\dset{e_n}{n\in\mathbb N}$, where $e_n$ denotes the sequence whose $n$-th coordinate is $1$ and all remaining coordinates are $0$, considered as a subset of the sequence Banach spaces $c$, $c_0$ or $l^p$ for $p\in[1,+\infty]$, does not have the weak Carath\'eodory property.

The same argument also shows that the set $A$ enlarged by $\{0\}$ also does not have the weak Carath\'eodory property.
\end{remark}

So far, we have been concerned with the weak Carath\'eodory property. We now introduce a stronger notion, which explains the terminology adopted above.

\begin{definition}\label{def:caratheodory}
We say that a non-empty subset $A$ of a Banach space has the 
\emph{Carath\'eodory property} if there exists a positive integer $m$ 
such that the convex hull of $A$ consists precisely of all convex 
combinations of at most (not necessarily distinct) $m$ points of $A$.
\end{definition} 

\begin{remark}
Note that in Definition~\ref{def:caratheodory} we could equivalently require all convex combinations to use exactly $m$ points by including extra elements with zero coefficients.
\end{remark}

As the name suggests, the Carath\'eodory property is stronger than its weak counterpart.

\begin{proposition}\label{prop:Cp_implies_wCp}
Every set that has the Carath\'eodory property \textup(in particular, every convex and every finite-dimensional set\textup) also enjoys the weak Carath\'eodory property.
\end{proposition}

\begin{proof}
Let us fix a non-empty set $A$ whose convex hull consists precisely of all convex combinations of exactly $m$ elements. Then, for any $x\in \conv A$ there exist points $a_1, \ldots,a_m \in A$ and coefficients $\lambda_1,\ldots,\lambda_m \in [0,1]$ with $\sum_{k=1}^{m}\lambda_k = 1$ such that
\[
x = \sum_{k=1}^{m} \lambda_k a_k.
\]
If all coefficients $\lambda_k$ were strictly less than $\frac{1}{m}$, their sum could not equal $1$. Hence, at least one coefficient must be greater than or equal to $\frac{1}{m}$.
By Proposition~\ref{prop:wCp_equivalent}, this shows that $A$ has the weak Carath\'eodory property with parameter $\lambda=\frac{1}{m}$.
\end{proof}

Now, we present an example of a set that has the weak Carath\'eodory
property but does not satisfy the Carath\'eodory property.

\begin{example}\label{ex:wCp_not_Cp}
Consider the Banach space $c_0$ of real null sequences endowed with the supremum norm, and its subset $A$ consisting of all finite sums of distinct unit vectors, including the empty sum which yields the zero sequence. 

We begin by showing that $\conv A$ coincides with the set $B$ of all sequences $(\xi_k)_{k \in \mathbb{N}}$ having only finitely many non-zero terms and satisfying $\xi_k \in [0,1]$ for every $k \in \mathbb N$. The inclusion $\conv A \subseteq B$ is clear. We prove the reverse inclusion by induction on the number of non-zero terms.

Let $x = (\xi_k)_{k \in \mathbb N} \in B$. If $x$ has at most one non-zero
term, then clearly $x \in \conv A$. Now, assume that every sequence in $B$ with at most $n$ non-zero terms belongs to $\conv A$, and let $x \in B$ have $n+1$ non-zero terms. For simplicity, assume these occur in the first $n+1$ coordinates. Set $\lambda:=\min_{1\leq k \leq n+1} \xi_k$. If $\lambda=1$, then 
\[
 x=\sum_{k=1}^{n+1} e_k \in A \subseteq \conv A.
\]
On the other hand, if $\lambda < 1$, we may write
\[
 x = \sum_{k=1}^{n+1} \xi_k e_k = \lambda \sum_{k=1}^{n+1} e_k + (1-\lambda) \sum_{k=1}^{n+1} \frac{\xi_k-\lambda}{1-\lambda} e_k. 
\] 
The first sum on the right-hand side of this decomposition belongs to $A$. The second sum belongs to $\conv A$ by the inductive assumption, since it involves at most $n$ non-zero terms and all coefficients lie in $[0,1]$. It follows that $x \in \conv A$, and hence $\conv A = B$.

Now, we prove that $A$ has the weak Carath\'eodory property. Let $x \in \conv A$. Then, there exist $n \in \mathbb N$ and $\xi_k \in [0,1]$ for $k=1,\ldots,n$ such that
\[
x = \sum_{k=1}^n \xi_k e_k.
\]
Set $I:=\dset{k \in \mathbb N}{\xi_k > \frac{1}{2}}$, and express $x$ as
\[
 x = \frac{1}{2}\Biggl( \sum_{k \in I} (2\xi_k - 1)e_k + \sum_{k \notin I} 2\xi_k e_k\Biggr) + \frac{1}{2} \sum_{k \in I} e_k \in \frac{1}{2}\conv A + \frac{1}{2}A;
\] 
if the set $I$ is empty, we take the corresponding sum to be zero. This shows that $\conv A \subseteq \frac{1}{2}A + \frac{1}{2} \conv A$.

It remains to show that $A$ does not have the Carath\'eodory property. Fix $n \in \mathbb{N}$ and consider
\[
x: = \sum_{k=1}^{2^n} \frac{k}{2^n} e_k \in \conv A.
\]
Suppose that $x$ can be written as a convex combination of at most $n$ elements of $A$, that is,
\[
x = \sum_{j=1}^n \lambda_j a_j,
\]
for some $a_j \in A \setminus \{0\}$ and $\lambda_j \geq 0$ with $0 < \sum_{j=1}^n \lambda_j \leq 1$. Each $a_j$ is a finite sum of distinct unit vectors, so choosing $m \in \mathbb N$ large enough, we may write all of them in the form
\[
 a_j = \sum_{k=1}^m \beta_k^j e_k, \quad j=1,\ldots,n,
\]
where $\beta_k^j \in \{0,1\}$. Substituting this into the expression for $x$, we obtain
\[
 x=\sum_{j=1}^n \lambda_j a_j = \sum_{j=1}^n \lambda_j \Biggl(\sum_{k=1}^m \beta_k^j e_k\Biggr) = \sum_{k=1}^m \Biggl( \sum_{j=1}^n \lambda_j \beta_k^j\Biggr) e_k.
\]
Comparing coefficients with the definition of $x$, we see that necessarily
$m \geq 2^n$ and
\[
 \sum_{j=1}^n \lambda_j \beta_k^j = \frac{k}{2^n} \ \text{for $k=1,\ldots,2^n$.}
\]
In particular, for every $k \in \{1,\ldots,2^n\}$ at least one of the coefficients $\beta_k^1, \ldots, \beta_k^n$ must be non-zero. Since there are only $2^n - 1$ non-zero binary vectors of length $n$, the pigeonhole principle implies that at least two of the $2^n$ vectors $(\beta_k^1, \ldots, \beta_k^n)$ must coincide. But then the corresponding left-hand sides of the above system are equal, while the right-hand sides are distinct -- a contradiction.

Thus, representing $x$ as a convex combination of elements of $A$ requires at least $n+1$ terms. This implies that $A$ does not have the Carath\'eodory property.
\end{example}

To conclude our general study of the Carath\'eodory property, let us present a result similar in spirit to Proposition~\ref{prop:weak_caratheodory_properties}.

\begin{proposition}\label{prop:caratheodory_properties}
If non-empty subsets $A,B$ of a Banach space have the Carath\'eodory property, then so do $A \cup B$ and $A + B$.
\end{proposition}

\begin{proof}
Assume that any element of $\conv A$ and $\conv B$ can be expressed as a convex combination 
of exactly (not necessarily distinct) $m$ and $n$ points, respectively.

Since $\conv(A+B)=\conv A + \conv B$, for every $x \in \conv(A+B)$ we can write
\[
 x = \sum_{k=1}^m \alpha_k a_k + \sum_{l=1}^n \beta_l b_l,
\] 
for some $a_1,\dots,a_m \in A$, $b_1,\dots,b_n \in B$, and $\alpha_1,\dots,\alpha_m, \beta_1,\dots,\beta_n \in [0,1]$ with  $\sum_{k=1}^m \alpha_k = \sum_{l=1}^n \beta_l = 1$. It follows that
\[
 x = \sum_{k=1}^m \sum_{l=1}^n \alpha_k \beta_l \,(a_k + b_l),
\]
and $\sum_{k=1}^m \sum_{l=1}^n \alpha_k \beta_l = 1$, which shows that $A+B$ also satisfies the Carath\'eodory property with $m \cdot n$.

Now, let us show that the same holds for $A \cup B$ with the parameter $m+n$. Take any $x \in \conv(A \cup B)$. Then,
\[
 x = \sum_{k=1}^p \alpha_k y_k,
\]
where $y_1,\dots,y_p \in A \cup B$ and $\alpha_1,\dots,\alpha_p \in [0,1]$ satisfy 
$\sum_{k=1}^p \alpha_k = 1$. Let $I$ denote the set of indices $k$ such that $y_k \in A$, 
and set $J := \{1,\dots,p\} \setminus I$. If both $I$ and $J$ are non-empty, we can write
\[
 x = \sum_{l \in I} \alpha_l \Biggl( \sum_{k \in I} \frac{\alpha_k}{\sum_{l \in I} \alpha_l} y_k \Biggr)
   + \sum_{l \in J} \alpha_l \Biggl( \sum_{k \in J} \frac{\alpha_k}{\sum_{l \in J} \alpha_l} y_k\Biggr).
\]
Since
\[
 \sum_{k \in I} \frac{\alpha_k}{\sum_{l \in I} \alpha_l} y_k \in \conv A
 \quad \text{and} \quad
 \sum_{k \in J} \frac{\alpha_k}{\sum_{l \in J} \alpha_l} y_k \in \conv B,
\]
it follows that $x$ can be expressed as a convex combination of $m+n$ points of $A \cup B$. Note that if either $I$ or $J$ is empty, the claim is immediate. Thus, $A \cup B$ has the Carath\'eodory property.
\end{proof}

\begin{remark}
The (weak) Carath\'eodory property is not necessarily inherited by subsets. For instance, the closed unit ball in the Banach space $c_0$ of null real sequences is convex and trivially has the Carath\'eodory property, but its subset $A:=\dset{e_n \in c_0}{n \in \mathbb N}$ does not (see Remark~\ref{rem:no_weak_C}).
\end{remark}

\begin{remark}\label{rem:minkowski_difference_intersection}
Similarly, the (weak) Carath\'eodory property is generally not preserved under intersections, nor under the Minkowski difference of sets. For completeness, recall that for two subsets $A$ and $B$ of a Banach space $E$, their Minkowski difference $A \dotminus B$ is defined as $\dset{x \in E}{x+B\subseteq A}$. It can be shown that, equivalently, $A \dotminus B= \bigcap_{b \in B} (A - b)$ (see, for example,~\cite{Schneider}*{p.~146}).

Consider the Banach space $c_0$ of real null sequences endowed with the supremum norm, and its subset $A:=\dset{e_n \in c_0}{n \in \mathbb N} \cup \{0\}$, where $e_n$ denotes the $n$-th unit vector. Next, set $B:=\ball_{c_0}(3e_1,1) \cup \ball_{c_0}(-5e_1,1) \cup A$.

We claim that $B$ enjoys the Carath\'eodory property. To see why, consider the set
\[
 C=\dset[\big]{(\xi_k)_{k \in \mathbb N} \in c_0}{\text{$\xi_1 \in [-6,4]$ and $\xi_k \in [-1,1]$ for $k\geq 2$}}.
\]
It is easy to check that $C$ is convex and contains $B$. Moreover, for any 
$x = (\xi_k)_{k \in \mathbb{N}} \in C$, we can find $\lambda \in [0,1]$ such that
$\xi_1=-6\lambda +4(1-\lambda)$. This allows us to write
\[
 x = \lambda(-6,\xi_2,\xi_3,\ldots) + (1-\lambda)(4,\xi_2,\xi_3,\ldots) \in \conv B. 
\] 
This observation directly shows that $C=\conv B$ and that $B$ satisfies the Carath\'eodory property. By Proposition~\ref{prop:caratheodory_properties}, the translated set $B - 3e_1$ also has the Carath\'eodory property. 

Yet, a simple calculation gives $B \cap (B - 3e_1) = A$ and $B \dotminus \{0, 3e_1\} = A$, illustrating that intersections and Minkowski differences of sets having the Carath\'eodory property do not necessarily have the property themselves (cf. Remark~\ref{rem:no_weak_C}).
\end{remark}

The importance of sets with the (weak) Carath\'eodory property is highlighted by the following theorem. Although the proof is fairly straightforward, the result itself is still quite significant.

\begin{theorem}\label{thm:axiom_iii}
Let $E$ be a Banach space and let $\nu \colon \mathcal B_E \to [0,+\infty)$ be a set function that satisfies axioms~\ref{def:axiomatic_MNC_ii} and~\ref{def:axiomatic_MNC_iv}. Then, $\nu(\conv A)=\nu(A)$ for every set $A \in \mathcal B_E$ that has the weak Carath\'eodory property.  
\end{theorem}

\begin{proof}
Since $A$ has the weak Carath\'eodory property, there exists $\lambda \in (0,1]$ such that $\conv A \subseteq \lambda A + (1-\lambda)\conv A$. Hence, in view of axioms~\ref{def:axiomatic_MNC_ii} and~\ref{def:axiomatic_MNC_iv}, we obtain 
\[
 \nu(\conv A) \leq \lambda \nu(A) + (1-\lambda) \nu(\conv A),
\]
which immediately yields $\nu(\conv A)\leq \nu(A)$. The opposite inequality follows from axiom~\ref{def:axiomatic_MNC_ii} and the fact that $A\subseteq \conv A$.
\end{proof}

From the classical Carath\'eodory’s theorem (see, for example,~\cite{Schneider}*{Theorem~1.1.4}) and Theorem~\ref{thm:axiom_iii} we immediately obtain the following corollary.

\begin{corollary}
If a Banach space is finite-dimensional, axiom~\ref{def:axiomatic_MNC_iii} is a consequence of the remaining axioms.
\end{corollary}

A natural question is whether there exist Banach spaces with bounded, non-empty subsets that lack the weak Carath\'eodory property but for which the conclusion of Theorem~\ref{thm:axiom_iii} still holds; the next example provides a positive answer.

\begin{example}\label{ex:without_caratheodory_but_axiom_iii_works}
Consider the Banach space $c_0$ of real null sequences endowed with the supremum norm $\norm{\cdot}_\infty$, and its subset $A:=\dset{e_n \in c_0}{n \in \mathbb N} \cup \{0\}$, where $e_n$ denotes the $n$-th unit vector. As noted in Remark~\ref{rem:no_weak_C}, $A$ does not have the weak Carath\'eodory property. Our aim is to prove that for any set function $\nu \colon \mathcal{B}_{c_0} \to [0,+\infty)$ satisfying axioms~\ref{def:axiomatic_MNC_ii} and~\ref{def:axiomatic_MNC_iv}, we have $\nu(\conv A) = \nu(A)$.

For each $m \in \mathbb{N}$, define
\[
 A_m := \sum_{k=1}^{2^m} \frac{1}{2^m} A.
\]
It is easy to see that $A \subseteq A_m \subseteq A_{m+1} \subseteq \conv A$, and that $\nu(A_m) = \nu(A)$.

We claim that the Hausdorff distance between $A_m$ and $\conv A$ satisfies the estimate
\[
 d_H(A_m, \conv A) \leq \frac{1}{2^m} \quad \text{for any $m \in \mathbb N$}.
\] 
To see this, fix $m \in \mathbb{N}$ and take any $x \in \conv A$. Write
\[
 x = \sum_{k=1}^n \lambda_k e_k
\]
with the coefficients $\lambda_1,\dots,\lambda_n \in [0,1]$ such that $\sum_{k=1}^\infty \lambda_k \leq 1$. If $\lambda_j = 1$ for some $j$, then $x = e_j \in A \subseteq A_m \subseteq A_m + \ball_{c_0}(0, 2^{-m})$. Otherwise, for each $k$ choose $l_k \in \{0,1,\dots,2^m-1\}$ such that
\[
 \frac{l_k}{2^m} \leq \lambda_k < \frac{l_k + 1}{2^m}.
\]
Clearly, $l_1+\ldots+l_n \leq 2^m(\lambda_1 + \ldots + \lambda_n) \leq 2^m$. Define
\[
 y := \sum_{k=1}^n \frac{l_k}{2^m} e_k.
\]
We can treat $y$ as a sum of $2^m$ terms with coefficients $2^{-m}$, by filling in the remaining $2^m - (l_1 + \dots + l_n)$ places with $2^{-m} \cdot 0$ if necessary. Hence, $y \in A_m$. Moreover,
\[
 \norm{x - y}_\infty = \sup_{1 \leq k \leq n} \Bigl(\lambda_k - \frac{l_k}{2^m}\Bigr) \leq \frac{1}{2^m}.
\]
Therefore, $x = y+(x-y) \in A_m + \ball_{c_0}(0, 2^{-m})$. Since $A_m \subseteq \conv A \subseteq \conv A + \ball_{c_0}(0, 2^{-m})$, it follows that 
\[
 d_H(A_m, \conv A) \leq \frac{1}{2^m}.
\]

Finally, by the continuity of $\nu$ with respect to the Hausdorff distance 
(see Proposition~\ref{prop:continuity_nu}) and the equality $\nu(A)=\nu(A_m)$, we obtain $\nu(A) = \nu(\conv A)$.
\end{example}

\begin{remark}
We would like to emphasize that the approach used in Example~\ref{ex:without_caratheodory_but_axiom_iii_works} works because of the geometry of $c_0$. If we take the same set $A$, but now regard it as a subset of $l^1$ -- the Banach space of all absolutely summable sequences equipped with its standard norm $\norm{\cdot}_1$ -- the situation changes completely. In this case, it is easy to see that $d_H(A_m, \conv A)=1$ for all $m \in \mathbb N$. 
\end{remark}

Following the ideas of Example~\ref{ex:without_caratheodory_but_axiom_iii_works}, Theorem~\ref{thm:axiom_iii} extends naturally to relatively compact sets.

\begin{theorem}\label{thm:axiom_iii_added}
Let $E$ be a Banach space and let $\nu \colon \mathcal B_E \to [0,+\infty)$ be a set function that satisfies axioms~\ref{def:axiomatic_MNC_ii} and~\ref{def:axiomatic_MNC_iv}. Then, $\nu(\conv A)=\nu(A)$ for every set $A \in \mathcal K_E$.  
\end{theorem}

\begin{remark}
Before proceeding to the proof of Theorem~\ref{thm:axiom_iii_added}, note that this result is trivial when $\ker \nu = \mathcal{K}_E$, since in this case $\nu(\conv A)=\nu(A)=0$ by Mazur's theorem (see, for example,~\cite{BBK}*{Corollary~1.1.22} or~\cite{Diestel1984}*{Exercise~1, p.~4}). In general, however, the kernel of the set function $\nu$ may be strictly contained in $\mathcal{K}_E$, and the claim is then no longer immediate.
\end{remark}

\begin{proof}[Proof of Theorem~\ref{thm:axiom_iii_added}]
It is well-known that if $A\subseteq E$ is a relatively compact set, then for every positive integer $m$ there exist points $a_1, \ldots, a_{n_m} \in A$ such that
\[
 A \subseteq \bigcup_{k=1}^{n_m} \ball_E\bigl(a_k, \tfrac{1}{m}\bigr).
\]
Set $A_m:=\dset{a_k}{k=1,\ldots,n_m}$. Since each $A_m$ is finite-dimensional, Theorem~\ref{thm:axiom_iii} implies that $\nu(A_m) = \nu(\conv A_m)$. Moreover, for each $m \in \mathbb N$ we have
\[
 A \subseteq A_m + \ball_E\bigl(0, \tfrac{1}{m}\bigr) \quad \text{and} \quad
 A_m \subseteq A \subseteq A + \ball_E\bigl(0, \tfrac{1}{m}\bigr),
\]
so that $d_H(A, A_m) \leq \frac{1}{m}$. Similarly, 
\[
 \conv A \subseteq \conv\bigl(A_m + \ball_E(0, \tfrac{1}{m})\bigr) 
 \subseteq \conv A_m + \ball_E\bigl(0, \tfrac{1}{m}\bigr)
\]
and
\[
 \conv A_m \subseteq \conv A \subseteq \conv A + \ball_E\bigl(0, \tfrac{1}{m}\bigr),
\]
which gives $d_H(\conv A, \conv A_m) \leq \frac{1}{m}$. Hence, the sequences $(A_m)_{m \in \mathbb N}$ and $(\conv A_m)_{m \in \mathbb N}$ converge to $A$ and $\conv A$, respectively, in the Hausdorff distance, and by continuity of $\nu$ we obtain
\[
 \nu(A) = \lim_{m \to \infty} \nu(A_m) = \lim_{m \to \infty} \nu(\conv A_m) = \nu(\conv A).
\]
This completes the proof.
\end{proof}

\subsection{Axiom~\ref{def:axiomatic_MNC_iv}} 
To show that axiom~\ref{def:axiomatic_MNC_iv} is independent of the others let us consider the set function $\nu \colon \mathcal B_{E} \to [0,+\infty)$ given by $\nu(A) = \sqrt{\norm{A}}$. Since the norm function $A \mapsto \norm{A}$ is an axiomatic measure of non-compactness it is straightforward to check that $\nu$ satisfies axioms~\ref{def:axiomatic_MNC_i}--\ref{def:axiomatic_MNC_iii} as well as axiom~\ref{def:axiomatic_MNC_v}. However, $\nu$ fails to satisfy axiom~\ref{def:axiomatic_MNC_iv}. To see why, let us fix a point $x\in E$ of norm $1$ and let us look at two singletons $A = \{x\}$, $B \coloneqq \{0\}$, and the number $\lambda \coloneqq \frac{1}{4}$. Then, we find that $\nu(\lambda A + (1-\lambda)B)=\frac{1}{2}$, while $\lambda\nu(A)+(1-\lambda)\nu(B)=\frac{1}{4}$.

\subsection{Axiom~\ref{def:axiomatic_MNC_v}} 
\label{sec:v}

In a recent paper~\cite{MR4686571}, X.~Chen and L.~Cheng addressed several questions concerning the representation and construction of axiomatic measures of non-compactness, as well as the existence of inequivalent regular measures. In Section~8 of that paper they stated, without proof, a result (Theorem~8.3) claiming that in every infinite-dimensional Banach space axiom~\ref{def:axiomatic_MNC_v} is independent of the remaining axioms. They noted that the details were contained in a preprint entitled \emph{On the fullness of measure of non-compactness} by J.~Bana\'s, X.~Chen, L.~Cheng, and W.~He; however, despite our efforts, we were unable to locate this manuscript. Instead, the material referred to in~\cite{MR4686571}*{Section~8} appears in~\cite{ChenChengHe2023}. Since that paper is written entirely in Chinese, we restate the result here and supply a sketch of the proof for the convenience of the reader. 

\begin{theorem}[cf.~\cite{MR4686571}*{Theorem~8.3} and \cite{ChenChengHe2023}*{Theorem~4.1}]\label{thm:axiom_v_failure}
In each infinite-dimensional Banach space $E$ there exists a set function $\nu \colon \mathcal B_E \to [0,+\infty)$ that satisfies axioms~\ref{def:axiomatic_MNC_i}--\ref{def:axiomatic_MNC_iv}, yet fails to satisfy axiom~\ref{def:axiomatic_MNC_v}.  
\end{theorem}

\begin{proof}
In the proof we will rely on several notions and results from functional analysis (see, for example, in~\cite{Diestel1984}*{Chapter~V}).

A classical theorem of Mazur ensures that $E$ contains a closed infinite-dimensional subspace $F$ admitting a Schauder basis $(x_n)_{n \in \mathbb{N}}$ consisting of unit vectors. Hence, every element $x \in F$ has a unique expansion
\[
  x = \sum_{k=1}^\infty \xi_k(x) x_k,
\]
where $(\xi_k(x))_{k \in \mathbb{N}}$ is a scalar sequence. The same theorem further guarantees that the associated linear projections $P_n \colon F \to F$, defined by
\[
  P_n(x) = \sum_{k=1}^n \xi_k(x)\, x_k,
\]
form a uniformly bounded family of operators, with operator norms bounded by $2$.

Fix an element $x \in F$. Observe that
\begin{align*}
 \sum_{k=2}^\infty \frac{\abs{\xi_k(x)}}{2^{k+2}} \leq \frac{1}{4}\sup_{k\geq 2}\abs{\xi_k(x)} = \frac{1}{4} \sup_{k \geq 2} \norm{P_k(x)-P_{k-1}(x)}_E \leq \norm{x}_E.
\end{align*}
Hence, the series
\[
  \sum_{k=1}^\infty \frac{\xi_k(x)}{2^{k+2}}\, x_k
\]
is absolutely convergent in $E$. Consequently, the expression
\[
 \abs{x}_F := \norm[\Bigg]{\sum_{k=1}^\infty \frac{\xi_k(x)}{2^{k+2}} x_k}_E
\]
is well-defined. This defines a norm on $F$, which is not equivalent to the norm $\norm{\cdot}_E$ inherited from $E$, and satisfies $\abs{x}_F \leq \norm{x}_E$ for all $x \in F$. Using $\abs{\cdot }_F$ we can define a new norm on $E$, inequivalent to $\norm{\cdot}_E$, by setting  
\[
 \abs{x}_E := \inf\dset[\big]{\abs{y}_F + \norm{x-y}_E}{y \in F}.
\]
By construction, this norm agrees with $\abs{\cdot}_F$ on $F$. The verification of all norm properties for $\abs{\cdot}_F$ and $\abs{\cdot}_E$ is straightforward.

Finally, define $\nu \colon \mathcal B_E \to [0,+\infty)$ by $\nu(A):=\sup_{x,y\in A}\abs{x-y}_E$. The set mapping $\nu$ satisfies axioms~\ref{def:axiomatic_MNC_i}--\ref{def:axiomatic_MNC_iv} (cf. examples in Section~\ref{sec:aMNCs}). However, $\nu$ fails to satisfy axiom~\ref{def:axiomatic_MNC_v}. To see this, consider the non-increasing sequence of bounded sets $A_n:=\{x_n,x_{n+1},\ldots\} \subseteq F$, where $n \in \mathbb N$. Each $A_n$ is closed in $(E,\norm{\cdot}_E)$. Indeed, if a subsequence $(x_{n_k})_{k \in \mathbb{N}} \subseteq A_n$ converged to some $x \in E$, then $x \in F$ since $F$ is closed in $E$. By uniqueness of the Schauder expansion, for any fixed $l \in \mathbb N$ and all sufficiently large $k$ we would have $\xi_l(x_{n_k})=0$, and the continuity of the coefficient functionals $\xi_l$ would then imply that $\xi_l(x) = 0$ for all $l \in \mathbb N$, so that $x = 0$. This is impossible, as $\norm{x_{n_k}}_E = 1$ for all $k \in \mathbb N$. 

As a byproduct, this also implies that the intersection $\bigcap_{n=1}^\infty A_n$ is empty. However, for each $n \in \mathbb{N}$, we have
\begin{align*}
 \nu(A_n) &= \sup_{i,j\geq n}\abs{x_i-x_j}_E=\sup_{i,j\geq n}\abs{x_i-x_j}_F\\
& = \sup_{i,j\geq n}\abs[\bigg]{\frac{1}{2^{i+2}} - \frac{1}{2^{j+2}}} \leq \frac{1}{2^{n+1}} \to 0 \ \text{as $n \to +\infty$}.
\end{align*}       
The proof is complete
\end{proof}

\begin{remark}
Note that, by definition, the set function $\nu$ introduced in the proof of Theorem~\ref{thm:axiom_v_failure} is also sub-additive and homogeneous, that is, $\nu(A+B)\leq \nu(A)+\nu(B)$ and $\nu(\alpha A)=\abs{\alpha} \nu(A)$ for all $A,B \in \mathcal B_E$ and $\alpha \in \mathbb R$.
\end{remark}

Let us observe that the set function $\nu$, whose existence is guaranteed by Theorem~\ref{thm:axiom_v_failure}, enables the construction of a wide class of set functions satisfying all the axioms except axiom~\ref{def:axiomatic_MNC_v}. Indeed, for any convex, strictly increasing function $h \colon [0,+\infty) \to \mathbb R$ that is continuous at $0$ and satisfies $h(0)=0$, the composition $h \circ \nu$ defines such a set function.

In what follows, we present another such set function, which possesses all the relevant properties of~$\nu$ or $h \circ \nu$ while being essentially different from them.

\begin{example}\label{ex:independence_of_axiom_v}
Consider the Banach space $C[0,1]$ of all continuous real-valued functions defined on the interval $[0,1]$, equipped with the supremum norm $\norm{\cdot}_\infty$.

For each function $f \in C[0,1]$ we define its average value by
\[
 m(f) \coloneqq \int_0^1 f(t)\,\textup dt.
\]
Using this, we introduce the set function $\nu \colon \mathcal B_{C[0,1]} \to [0,+\infty)$ by setting
\[
 \nu(A) \coloneqq \sup_{f \in \conv A} \biggl(\int_0^1 \abs{f(t)-m(f)}^2\,\textup dt \biggr)^{\frac{1}{2}}\!.
\]
The kernel of $\nu$ is non-empty. In fact, a simple argument shows that $\ker \nu$ consists exactly of the bounded subsets of $C[0,1]$ whose elements are constant functions. 

The fact that $\nu$ satisfies axioms~\ref{def:axiomatic_MNC_ii} and~\ref{def:axiomatic_MNC_iii} follows directly from its definition and elementary properties of the convex hull. Now, let us check that $\nu$ satisfies axiom~\ref{def:axiomatic_MNC_iv}.  Using the identity~\eqref{eq:convex} we see that any function $h \in \conv(\lambda A + (1-\lambda) B)$ can be expressed as $h=\lambda f + (1-\lambda)g$ for some $f \in \conv A$ and $g \in \conv B$. Thus, by linearity of the integral and the Minkowski inequality, we obtain
\begin{align*}
& \biggl(\int_0^1 \abs{h(t)-m(h)}^2 \,\textup dt\biggr)^{\frac{1}{2}}\\
&\qquad \leq  \lambda \biggl(\int_0^1 \abs{f(t)-m(f)}^2 \,\textup dt\biggr)^{\frac{1}{2}} +(1-\lambda)\biggl(\int_0^1 \abs{g(t)-m(g)}^2 \,\textup dt\biggr)^{\frac{1}{2}}\!.
\end{align*}
In other words,
\[
\biggl(\int_0^1 \abs{h(t)-m(h)}^2 \,\textup dt\biggr)^{\frac{1}{2}} \leq \lambda \nu(A) + (1-\lambda)\nu(B)
\]
for every $h \in \conv(\lambda A + (1-\lambda) B)$. Therefore, taking the supremum over all such $h$, we get
\[
 \nu(\lambda A + (1-\lambda) B) \leq \lambda \nu(A) + (1-\lambda)\nu(B),
\]
which shows that $\nu$ satisfies axiom~\ref{def:axiomatic_MNC_iv}.

Finally, we will show that $\nu$ does not satisfy axiom~\ref{def:axiomatic_MNC_v}. 
For each $n \in \mathbb{N}$, let us define the set
\[
 A_n  :=  \dset[\big]{f \in C[0,1]}{\text{$\norm{f}_\infty \leq 1$, $f(0)=0$, and $f(t)=1$ if $t \in [\tfrac{1}{n},1]$}}.
\]
The idea of using the sets $A_n$ comes from the paper~\cite{ARGF}, where, in the
course of studying fixed point theorems, the authors introduced a set function
that satisfies axiom~\ref{def:axiomatic_MNC_v}, provided the underlying Banach
space has sufficiently nice geometric properties (see~\cite{ARGF}*{Example~2.3 and Proposition~2.4}). 

Each $A_n$ is convex, bounded, and closed, and together they form a non-increasing sequence. Moreover, the intersection of these sets is empty. Indeed, suppose there existed a function $f_\ast \in \bigcap_{n =1}^\infty A_n$. Then, for every $n \in \mathbb N$ we would have $f_\ast(t) = 1$ on $[\tfrac{1}{n},1]$, so in particular $f_\ast(t) = 1$ for all $t \in (0,1]$.  But $f_\ast(0)=0$ by the definition of the sets $A_n$, which contradicts the continuity of $f_\ast$ at $0$. Thus, $\bigcap_{n=1}^\infty A_n=\emptyset$. So, all that is left is to show that $\nu(A_n) \to 0$ as $n \to +\infty$. Fix $n \in \mathbb N$, and let $f \in A_n$. We start by estimating the average value $m(f)$. Since $f(t)=1$ on the interval $[\frac{1}{n},1]$, we have
\begin{align*}
 m(f) &= \int_0^{\frac{1}{n}} f(t) \,\textup dt + \biggl(1-\frac{1}{n}\biggr).
\end{align*}
Now, because $-1 \leq f(t) \leq 1$ for $t \in [0,1]$, it follows that
\[
 -\frac{1}{n} \leq \int_0^{\frac{1}{n}} f(t) \,\textup dt \leq \frac{1}{n}.
\]
Combining these inequalities, we get the bound $-1 \leq 1- \frac{2}{n} \leq m(f) \leq 1$. To calculate the outer integral once again we look at the two parts of the interval $[0,1]$ separately. Since $\abs{f(t)-m(f)} \leq 2$ on $[0,\frac{1}{n}]$ we have
\[
 \int_0^{\frac{1}{n}} \abs{f(t)-m(f)}^2 \,\textup dt \leq \frac{4}{n}.
\]
On the other hand, for $t \in [\frac{1}{n},1]$ we know that $f(t)=1$, so $\abs{f(t)-m(f)} = 1-m(f) \leq \frac{2}{n}$. Therefore,
\[
 \int_{\frac{1}{n}}^1 \abs{f(t)-m(f)}^2 \,\textup dt \leq \frac{4}{n^2}\cdot \biggl(1-\frac{1}{n}\biggr) \leq \frac{4}{n^2} \leq \frac{4}{n}.
\]
Putting both pieces together, for any  $f \in A_n$ we get
\[
 \int_{\frac{1}{n}}^1 \abs{f(t)-m(f)}^2 \,\textup dt \leq \frac{8}{n}.
\]
Hence,
\[
 \nu(A_n)=\sup_{f \in A_n} \biggl(\int_0^1 \abs{f(t)-m(f)}^2 \,\textup dt \biggr)^{\frac{1}{2}} \leq \frac{3}{\sqrt{n}}.
\]
This shows that $\nu(A_n) \to 0$ as $n \to +\infty$, which completes the example.
\end{example}

\section{Concluding remarks}

The discussion carried out in this paper allows us to restate the definition of an axiomatic measure of non-compactness in a simpler but \emph{equivalent} form. In the case of finite-dimensional spaces, Definition~\ref{def:axiomatic_MNC} can be reformulated as follows.

\begin{definition}
Let $E$ be a finite-dimensional Banach space. A function $\mu \colon \mathcal B_E \to [0,+\infty)$ is called an \emph{axiomatic measure of non-compactness}, if
\begin{enumerate}[label=\textup{(a\arabic*)}]
 \item the family $\ker \mu  \coloneqq \dset{A \in \mathcal B_E}{\mu(A)=0}$ is non-empty,

 \item for any $A,B \in \mathcal B_E$ if $A \subseteq B$, then $\mu(A)\leq \mu(B)$,

 \item there exists $\lambda \in (0,1)$ such that $\mu(\lambda A + (1-\lambda)B)\leq \lambda \mu(A) + (1-\lambda)\mu(B)$ for all $A,B \in \mathcal B_E$.
\end{enumerate}
\end{definition}

In the case of infinite-dimensional spaces, Definition~\ref{def:axiomatic_MNC} is equivalent to the following formulation.

\begin{definition}
Let $E$ be an infinite-dimensional Banach space. A function $\mu \colon \mathcal B_E \to [0,+\infty)$ is called an \emph{axiomatic measure of non-compactness}, if
\begin{enumerate}[label=\textup{(a\arabic*)}]
 \item the family $\ker \mu  \coloneqq \dset{A \in \mathcal B_E}{\mu(A)=0}$ is non-empty,

 \item for any $A,B \in \mathcal B_E$ if $A \subseteq B$, then $\mu(A)\leq \mu(B)$,

 \item $\mu(\conv A)=\mu(A)$ for any $A \in \mathcal B_E$,

 \item there exists $\lambda \in (0,1)$ such that $\mu(\lambda A + (1-\lambda)B)\leq \lambda \mu(A) + (1-\lambda)\mu(B)$ for all $A,B \in \mathcal B_E$,

 \item if the sets $A_n$ are closed and $A_{n+1}\subseteq A_n$ for $n \in \mathbb N$ with $\mu(A_n)\to 0$, then their intersection $\bigcap_{n=1}^\infty A_n$ is non-empty.
\end{enumerate}
\end{definition}

We end the paper with an open question concerning the independence of axiom~\ref{def:axiomatic_MNC_iv}.

\begin{question}
Does every infinite-dimensional Banach space $E$ admit a set function 
$\nu \colon \mathcal B_E \to [0,+\infty)$ that satisfies all the axioms of 
Definition~\ref{def:axiomatic_MNC} except axiom~\ref{def:axiomatic_MNC_iii}?
\end{question}

\subsection*{Use of Large Language Models}
All mathematical results, proofs and arguments presented in this work were independently derived, fully formalized, and rigorously verified by the authors. No statement in this paper was accepted on the basis of model output alone, and the LLM was not used as a sole source of validated mathematical claims. Its role was limited to informal assistance in discussion and writing.


\begin{bibdiv}
  \begin{biblist}
\bib{AbletChengChengZhang2019}{article}{
  author={Ablet, E.},
  author={Cheng, L.},
  author={Cheng, Q.},
  author={Zhang, W.},
  title={Every Banach space admits a homogeneous measure of non-compactness not equivalent to the Hausdorff measure},
  journal={Sci. China Math.},
  volume={62},
  date={2019},
  pages={147--156 (with a corrected proof, Sci. China Math. \textbf{62} (2019), 2053--2056)},
}


\bib{AKPRS}{book}{
  author={Akhmerov, R. R.},
  author={Kamenskii , M. I.},
  author={Potapov, A. S.},
  author={Rodkina, A. E.},
  author={Sadovskii , B. N.},
  title={Measures of noncompactness and condensing operators},
  series={Operator Theory: Advances and Applications},
  volume={55},
  publisher={Birkh\"auser Verlag, Basel},
  date={1992},
}



\bib{ARGF}{article}{
  author={Ariza-Ruiz, D.},
  author={Garc\'ia-Falset, J.},
  title={Abstract measures of noncompactness and fixed points for nonlinear mappings},
  journal={Fixed Point Theory},
  volume={21},
  number={1},
  pages={47--66},
  date={2020},
}

\bib{ADL}{book}{
  author={Ayerbe Toledano, J. M.},
  author={Dom\'{\i }nguez Benavides, T.},
  author={L\'{o}pez Acedo, G.},
  title={Measures of noncompactness in metric fixed point theory},
  series={Operator Theory: Advances and Applications},
  volume={99},
  publisher={Birkh\"{a}user Verlag, Basel},
  date={1997},
}

\bib{BG80}{book}{
  author={Bana\'{s}, J.},
  author={Goebel, K.},
  title={Measures of noncompactness in Banach spaces},
  series={Lecture Notes in Pure and Applied Mathematics},
  volume={60},
  publisher={Marcel Dekker, Inc., New York},
  date={1980},
}

\bib{BanasMartinon1992}{article}{
  author={Bana\'s, J.},
  author={Mart\'{\i}n{\'o}n, A.},
  title={Measures of noncompactness in Banach sequence spaces},
  journal={Math. Slovaca},
  volume={42},
  date={1992},
  pages={497--503},
}

\bib{MR3289625}{book}{
   author={Bana\'s, J.},
   author={Mursaleen, M.},
   title={Sequence spaces and measures of noncompactness with applications
   to differential and integral equations},
   publisher={Springer, New Delhi},
   date={2014},
}

\bib{MR3642943}{collection}{
   title={Advances in nonlinear analysis via the concept of measure of
   noncompactness},
   editor={Bana\'s, J.},
   editor={Jleli, M.},
   editor={Mursaleen, M.},
   editor={Samet, B.},
   editor={Vetro, C.},
   publisher={Springer, Singapore},
   date={2017},
}

\bib{BBK}{book}{
  author={Borkowski, M.},
  author={Bugajewska, D.},
  author={Kasprzak, P.},
  title={Selected topics in nonlinear analysis},
  publisher={Juliusz Schauder Center for Nonlinear Studies},
  address={Toru\'n, Poland},
  series={Lecture Notes in Nonlinear Analysis},
  volume={19},
  date={2021},
}

\bib{BugajewskiGulgowski2019}{article}{
  author={Bugajewski, D.},
  author={Gulgowski, J.},
  title={On the characterization of compactness in the space of functions of bounded variation in the sense of Jordan},
  journal={J. Math. Anal. Appl.},
  volume={484},
  date={2019},
	pages={article no. 123752, 17 pages},
}

\bib{ChenCheng2021}{article}{
  author={Chen, X.},
  author={Cheng, L.},
  title={On countable determination of the Kuratowski measure of noncompactness},
  journal={J. Math. Anal. Appl.},
  volume={504},
  date={2021},
  pages={article no. 125370, 20 pages},
}

\bib{ChenCheng2023}{article}{
  author={Chen, X.},
  author={Cheng, L.},
  title={Representation of measures of noncompactness and its applications related to an initial-value problem in Banach spaces},
  journal={Sci. China Math.},
  volume={66},
  date={2023},
  pages={745--776},
}

\bib{MR4686571}{article}{
   author={Chen, X.},
   author={Cheng, L.},
   title={A review on several questions related to measure of
   noncompactness},
   journal={Pure Appl. Funct. Anal.},
   volume={8},
   date={2023},
   number={6},
   pages={1621--1650},
}

\bib{ChenChengHe2023}{article}{
  author={Chen, X.},
  author={Cheng, L.},
  author={He, W.},
  title={On nullity and fullness of measures of noncompactness},
  journal={Sci. China Math.},
  volume={53},
  number={12},
  pages={1577--1596},
  date={2023},
}

\bib{ChengChengShenTuZhang2018}{article}{
  author={Cheng, L.},
  author={Cheng, Q.},
  author={Shen, Q.},
  author={Tu, K.},
  author={Zhang, W.},
  title={A new approach to measures of noncompactness of Banach spaces},
  journal={Studia Math.},
  volume={240},
  date={2018},
  pages={21--45},
}


\bib{Diestel1984}{book}{
  author={Diestel, J.},
  title={Sequences and series in Banach spaces},
  publisher={Springer-Verlag},
  place={New York},
  date={1984},
  series={Graduate Texts in Mathematics},
  volume={92},
}

\bib{DominguezBenavides1986}{article}{
  author={Dom\'{\i}nguez Benavides, T.},
  title={Some properties of the set and ball measures of noncompactness and applications},
  journal={J. London Math. Soc.},
	number={2},
  volume={34},
  date={1986},
  pages={120--128},
}


\bib{Darbo}{article}{
  author={Darbo, G.},
  title={Punti uniti in trasformazioni a codominio non compatto},
  journal={Rend. Sem. Mat. Univ. Padova},
  volume={24},
  date={1955},
  pages={84--92},
}

\bib{GM65}{article}{
  author={Gol\cprime denshte\u {\i }n, L. S.},
  author={Markus, A. S.},
  title={On the measure of non-compactness of bounded sets and of linear operators},
  language={in Russian},
  conference={ title={Studies in Algebra and Math. Anal. (Russian)}, },
  book={ publisher={Izdat. ``Karta Moldovenjaske'', Kishinev}, },
  date={1965},
  pages={45--54},
}

\bib{GGM57}{article}{
  author={Gol\cprime denshte\u {i}n, L. S.},
  author={Gohberg, I.},
  author={Markus, A. S.},
  title={Investigation of some properties of bounded linear operators and of the connection with their $g$-norm},
  journal={Uchen. Zap. Kischinev. Gos. Univ.},
  volume={29},
  date={1957},
  pages={29\ndash 36},
  language={in Russian},
}

\bib{Krukowski2016}{article}{
  author={Krukowski, M.},
  title={Darbo-type theorem for quasimeasure of noncompactness},
  journal={arXiv preprint},
  date={2016},
  url={https://arxiv.org/abs/1605.05229},
}


\bib{Krukowski2017}{book}{
  author={Krukowski, M.},
  title={Compact families in spaces of continuous functions with the topology of uniform convergence},
	year={2017},
	publisher={PhD thesis, Łódź, Poland},	
}


\bib{Kuratowski}{article}{
  author={Kuratowski, K.},
  title={Sur les espaces complets},
  year={1930},
  journal={Fund. Math.},
  volume={15},
  number={1},
  pages={301--309},
}

\bib{MalletParetNussbaum2011a}{article}{
  author={Mallet-Paret, J.},
  author={Nussbaum, R.~D.},
  title={Inequivalent measures of noncompactness and the radius of the essential spectrum},
  journal={Proc. Amer. Math. Soc.},
  volume={139},
  date={2011},
  pages={917--930},
}

\bib{MalletParetNussbaum2011b}{article}{
  author={Mallet-Paret, J.},
  author={Nussbaum, R.~D.},
  title={Inequivalent measures of noncompactness},
  journal={Ann. Mat. Pura Appl.},
  volume={190},
  date={2011},
  pages={453--488},
}

\bib{Martinon1990}{article}{
   author={Martin\'{o}n, A},
   title={A system of axioms for measures of noncompactness},
   journal={Zeszyty Nauk. Politech. Rzeszowskiej Mat. Fiz.},
   number={10},
   date={1990},
   pages={133--143},
}

\bib{Sadovskii}{article}{
  author={Sadovskii, B. N.},
  title={On a fixed point principle},
  language={in Russian},
  journal={Funkcional. Anal. i Prilo\v {z}en.},
  volume={1},
  date={1967},
  number={2},
  pages={74--76},
}

\bib{Sadovskii1968}{article}{
  author={Sadovskii, B.~N.},
  title={Measures of noncompactness and condensing operators},
  journal={Problemy Mat. Anal. Slozh. Sistem},
  volume={2},
  date={1968},
  pages={89--119},
  language={Russian},
}


\bib{Sadovskii1972}{article}{
  author={Sadovskii, B.~N.},
  title={Limit-compact and condensing operators},
  journal={Uspekhi Mat. Nauk},
  volume={27},
  number={1(163)},
  date={1972},
  pages={81--146},
  translation={
    journal={Russian Math. Surveys},
    volume={27},
    number={1},
    pages={85--155}
  },
}


\bib{Schneider}{book}{
  author={Schneider, R.},
  title={Convex bodies\textup: The Brunn--Minkowski theory},
  series={Encyclopedia of Mathematics and its Applications},
  volume={151},
  publisher={Cambridge University Press},
  address={Cambridge},
  date={2014},
	edition={2nd edition},
}

 \end{biblist}
\end{bibdiv}
\end{document}